\documentclass{article}
\usepackage{subfig, enumerate}
\usepackage[section] {placeins}
\usepackage{amsthm}
\usepackage{amssymb,amsmath, mathdots}
\usepackage{anysize}
\usepackage{psfrag}
\usepackage{url}
\usepackage[normalem]{ulem}
\newtheorem{theorem}{Theorem}[section]
\newtheorem{definition}{Definition}[section]
\newtheorem{ex}{Example}[section]

\newtheorem{lemma}{Lemma}[section]

\numberwithin{table}{section}
\begin{document}

\title{A Proof of Euler's Theorem via Quandles}

\author{Ant\'onio Lages and Pedro Lopes\\
Center for Mathematical Analysis, Geometry, and Dynamical Systems, \\
Department of Mathematics, \\
Instituto Superior T\'{e}cnico, Universidade de Lisboa\\
1049-001 Lisbon, Portugal \\
\texttt{antonio.lages@tecnico.ulisboa.pt, pelopes@tecnico.ulisboa.pt}}
\maketitle
\bigbreak
\begin{abstract}
We prove Euler's theorem of number theory developing an argument based on quandles. A quandle is an algebraic structure whose axioms mimic the three Reidemeister moves of knot theory.
\end{abstract}

Keywords: Euler's theorem;  quandles; permutations

Mathematics Subject Classification 2020:  11A99; 20N99

\section{Introduction}\label{sec:intro}

In this article we present a different proof of Euler's theorem whose statement is as follows.

\begin{theorem}[Euler]\label{thm:Euler}
Let $n$ and $m$ be coprime positive integers. Then \[m\quad\mid\quad\big(n^{\varphi(m)}-1\big).\]
\end{theorem}

\bigbreak

In the statement of Theorem \ref{thm:Euler} above, $\varphi$ denotes Euler's totient function, i.e., the number of positive integers less than the argument that are coprime to the argument. Our proof is based on quandles. A quandle (\cite{Joyce}, \cite{Matveev}) is an algebraic structure whose axioms mimic the three Reidemeister moves in knot theory (\cite{lhKauffman}) and thus provides interesting invariants for knots (\cite{DL}). The three axioms of quandles (Section \ref{sec:quandles-euler}) are (i) idempotency, (ii) right-invertibility, (iii) self-distributivity. Right-invertibility allows us to conceive of each quandle as a sequence of permutations. Since  permutations can be broken down into  disjoint cycles, we  organize the study of quandles around the lengths these cycles may have (\cite{Lopes_Roseman}, \cite{Lages_Lopes}, \cite{Lages_Lopes_Vojtechovsky}). In the course of these considerations we arrive at a proof of Euler's Theorem. Specifically, in this article, we look at a particular class of quandles and for each one of them we consider its permutations. Then we calculate the length of  the disjoint cycles of these permutations along with the number of elements that make them up. Equating the length of these cycles with the number of elements comprised by these cycles, we prove Euler's Theorem for only one prime (Theorem \ref{thm:flt+}). Bringing together the contributions of the individual primes in the general case, we prove the full Euler's Theorem (Theorem \ref{thm:Euler}).

The organization of the current article is as follows. The background material is laid out in Section \ref{sec:quandles-euler}, as well as the proofs of the theorems. In Section \ref{sec:extra} we collect information about the quandle dealt with in Section \ref{sec:quandles-euler} to state some facts about its structure. In Section \ref{sec:final} we compare our proof with other algebraic proofs from the literature.

\section{Quandles and  proof of Euler's Theorem}\label{sec:quandles-euler}

\begin{definition}\label{def:quandle}
A {\it quandle}, $(X,\ast)$, is a set, $X$, equipped with a binary operation, $\ast$, such that, for any $a, b, c \in X$
\begin{enumerate}[(i)]
\item $a\ast a = a$ \quad -- \quad idempotency;

\item there exists a unique $x\in X$ such that $x\ast b = a$ \quad -- \quad right-invertibility;

\item $(a\ast b) \ast c = (a\ast c) \ast (b \ast c)$ \quad -- \quad self-distributivity.
\end{enumerate}
\end{definition}

Here are some examples of quandles.

\begin{ex}\label{ex:quandles}{\rm
\begin{enumerate}[1.]
\item Let $G$ be a group and consider the conjugation in the group, $a\ast b:=bab^{-1}$. Then $(G, \ast )$ is a quandle.
\item Let $n, m$ be  positive integers with $m$ invertible modulo $n$. Consider $a\ast b = ma + (1-m)b$ (mod $n$). Then $(\mathbf{Z}/n\mathbf{Z}, \ast)$ is a quandle. These quandles are called {\it linear Alexander quandles}. In the sequel, we call $m$ the {\rm parameter} of these quandles.
\end{enumerate}}
\end{ex}

\begin{definition}\label{def:rightTranslation}
For a quandle $(X,*)$ and for $y\in X$, we let $R_y: X\rightarrow X, x\mapsto x*y$ be the {\it right translation} by $y$ in $(X,*)$. Axiom (ii) in Definition \ref{def:quandle} guarantees that right translations are permutations of the elements of the quandle under study.
\end{definition}

We now prove a particular case of Euler's Theorem in which there is only one prime at stake.

\begin{theorem}\label{thm:flt+}
Let $p$ be a prime. Let $n, k$ be  positive integers with $n$ coprime with $p$. Then, $$p^k \mid (n^{p^k-p^{k-1}}-1) .$$
\end{theorem}
\begin{proof}
Since $n\cdot n^{p^k-1} \equiv  n^{p^k} \equiv 1$ (mod $n^{p^k}-1$), it follows that $n$ is invertible modulo $n^{p^k}-1$.
Consider, then, the linear Alexander quandle of order $n^{p^k}-1$ and parameter $n$ (cf. Example \ref{ex:quandles}, $2.$). The underlying set is $\mathbf{Z}/(n^{p^k}-1)\mathbf{Z}$ and, for any $a, b \in \mathbf{Z}$, $a\ast b = na+(1-n)b$ (mod $n^{p^k}-1$).
We will need the following lemma.

\begin{lemma}\label{lem:iterperm}
Let $a, b \in \mathbf{Z}$. Then, for any $l\in \mathbf{Z}^+$, $$R_b^{\,\,\, l}(a):=(\cdots ((a\ast \underbrace{b)\ast b) \cdots )\ast b}_{l}=n^{l}a - (n^{l}-1)b  .$$
\end{lemma}

\begin{proof}
This is a straightforward exercise in induction that we leave to the reader.
\end{proof}

Resuming the proof of Theorem \ref{thm:flt+}, setting $l=p^k$, we obtain $$R_b^{\,\,\, p^k}(a)= n^{p^k}a - (n^{p^k}-1)b \equiv a \qquad \text{ (mod } n^{p^k}-1) .$$
We note that,  modulo $n^{p^k}-1$,  the sequence $$(R_b^{\,\,\, l}(a)\, :\, l\in \mathbf{Z}_0^+)$$ is a periodic sequence starting at $a$ and consisting of copies of the cycle of $R_b$ that contains $a$. Then, $R_b^{\,\,\, p^k}(a) \equiv a$  implies that the length of the disjoint cycle in $R_b$ containing $a$, divides  $p^k$. The possibilities for these lengths are then $1, p, p^2, \dots , p^i, \dots , p^k$. We need to count the $a$'s that satisfy $R_b^{\,\,\, p^i}(a) \equiv a$, for the different $i$'s. We call an $a$ that satisfies $R_b^{\,\,\, p^i}(x) \equiv x$ but does not satisfy $R_b^{\,\,\, p^j}(x) \equiv x$, for $j<i$, a {\it proper solution} of $R_b^{\,\,\, p^i}(x) \equiv x$. We will need the following lemma.

\begin{lemma}\label{lem:propersols}
The number of proper solutions, modulo $n^{p^k}-1$, of $R_b^{\,\,\, p^i}(x) \equiv x$ is $n^{p^i}-n^{p^{i-1}}$, for $1\leq i \leq k$, and $n-1$ for $i=0$.
\end{lemma}
\begin{proof}
We use induction on $i$. For $i=0$, we are looking for the solutions of $x \equiv R_b^{\,\,\, 1}(x) = nx-(n-1)b$, i.e., $(n-1)(x-b)\equiv0$ (mod $n^{p^k}-1$). Now, $$n^{p^k}-1 = (n-1)(n^{p^k-1}+\dots + n + 1)$$ so the proper solutions, modulo $n^{p^k}-1$, are:
\begin{align*}
&x_0\equiv b\\
&x_1\equiv b+(n^{p^k-1}+n^{p^k-2}+ \cdots + n^2 + n + 1)\\
&x_2\equiv b+2(n^{p^k-1}+n^{p^k-2}+ \cdots + n^2 + n + 1)\\
& \vdots \\
&x_{n-2}\equiv b+(n-2)(n^{p^k-1}+n^{p^k-2}+ \cdots + n^2 + n + 1).
\end{align*}
There are then $n-1$ proper solutions for $i=0$.

For $i=1$, we are looking for the solutions of  $x \equiv R_b^{\,\,\, p}(x) = n^px-(n^p-1)b$, i.e., $(n^p-1)(x-b)\equiv 0$ (mod $n^{p^k}-1$). Now, $$n^{p^k}-1 = (n^p-1)((n^p)^{p^{k-1}-1}+\dots + n^p + 1)$$ so the solutions are, modulo $n^{p^k}-1$:
\begin{align*}
&x_0\equiv b\\
&x_1\equiv b+((n^p)^{p^{k-1}-1}+\dots + n^p + 1)\\
&x_2\equiv b+2((n^p)^{p^{k-1}-1}+\dots + n^p + 1)\\
& \vdots \\
&x_{n^p-2}\equiv b+(n^p-2)((n^p)^{p^{k-1}-1}+\dots + n^p + 1).
\end{align*}
There are then $n^p-1$ solutions for $i=1$. In order to obtain the number of proper solutions of $x \equiv  R_b^{\,\,\, p}(x)$ we subtract from $n^p-1$ the number of solutions of $x \equiv  R_b^{\,\,\, 1}(x)$. There are, then, $n^p-1-(n-1) = n^p -n$ proper solutions of $x \equiv  R_b^{\,\,\, p}(x)$.

We now assume that, modulo $n^{p^k}-1$, there is an integer $i\geq 1$ such that the number of proper solutions of $R_b^{\,\,\, p^j}(x) \equiv x$, for each $1\leq j \leq i-1$, is $n^{p^j}-n^{p^{j-1}}$, and $n-1$ for $j=0$. Modulo $n^{p^k}-1$, the equation  $R_b^{\,\,\, p^{i}}(x) \equiv  x$ is equivalent to $0\equiv (n^{p^{i}}-1)(x-b)$. Furthermore, $$n^{p^k}-1 = \big(n^{p^{i}}\big)^{p^{k-i}}-1=\big(n^{p^{i}}-1\big)\big(\big(n^{p^{i}}\big)^{p^{k-i}-1}+\big(n^{p^{i}}\big)^{p^{k-i}-2} + \dots + n^{p^{i}} + 1\big) .$$ Then, the solutions, modulo $n^{p^k}-1$, are: $$x_j\equiv b+j\big(\big(n^{p^{i}}\big)^{p^{k-i}-1}+\big(n^{p^{i}}\big)^{p^{k-i}-2} + \dots + n^{p^{i}} + 1\big) \qquad \text{ for }\quad 0\leq j \leq n^{p^{i}}-2 .$$ There are, then, $n^{p^i}-1$ such solutions. In order to obtain the number of proper solutions, we subtract from $n^{p^i}-1$, the number of proper solutions of the equations of the sort $R_b^{\,\,\, p^{j}}(x)\equiv x$, modulo $n^{p^k}-1$, for $0\leq j \leq i-1$: $$n^{p^{i}}-1 - \sum_{j=1}^{i-1}\big( n^{p^j}-n^{p^{j-1}} \big) - (n-1) = n^{p^{i}} - n^{p^{i-1}} .$$ This completes the proof of Lemma \ref{lem:propersols}.
\end{proof}

Resuming the proof of Theorem \ref{thm:flt+}, the number of proper solutions of $R_b^{\,\,\, p^k}(x) \equiv  x$, modulo $n^{p^k}-1$,  is $n^{p^{k}} - n^{p^{k-1}}$. But the set of these proper solutions coincides with the union of all disjoint cycles of length $p^k$ of $R_b$. Therefore, $$p^k \mid \big( n^{p^{k}} - n^{p^{k-1}}\big) .$$ Since $n^{p^{k}} - n^{p^{k-1}}=n^{p^{k-1}}\big( n^{p^k-p^{k-1}} -1 \big)$ and $n$ and $p$ are coprime, it follows that $$p^k \mid \big( n^{p^k-p^{k-1}} -1 \big) .$$ This completes the proof of Theorem \ref{thm:flt+}.
\end{proof}

We are now ready to prove Euler's Theorem. We rewrite the statement of Euler's Theorem here as Theorem \ref{thm:euler} below. In the statement of Theorem \ref{thm:euler} we emphasize the role of the individual primes in order to bring out the connection with Theorem \ref{thm:flt+}.

\begin{theorem}[Euler]\label{thm:euler}
Let $p_1, p_2, \dots , p_N$ be a sequence of distinct primes, $k_1, k_2, \dots , k_N$ a sequence of positive integers, and $n$ a positive integer coprime with each of the primes. Then $$\prod_{j=1}^N p_j^{k_j} \quad\mid \quad  \Bigg( n^{\,\,\displaystyle{\prod_{j=1}^N \big(p_j^{k_j}-p_j^{k_j-1}\big)}} -1\Bigg) .$$
\end{theorem}
\begin{proof}
Let $1\leq i \leq N$. By Theorem \ref{thm:flt+} we have $$p^{k_i} \,\mid\, \big( n^{p^{k_i}-p^{k_i-1}} -1 \big) .$$
Since $$m^{rs}-1 = (m^r)^s-1 = (m^r-1)\big((m^r)^{s-1}+ \cdots + m^r + 1\big) ,$$ it follows that $$(m^r-1) \,\mid \, (m^{rs}-1) ,$$ and so $$p^{k_i} \,\mid\, \big( n^{p^{k_i}-p^{k_i-1}} -1 \big)  \,\mid\, \Bigg( n^{\,\,\displaystyle{\prod_{j=1}^N \big(p_j^{k_j}-p_j^{k_j-1}\big)}} -1\Bigg) .$$ Moreover, since any two of the primes are distinct, $$\prod_{j=1}^N p_j^{k_j} \quad\Huge\mid\quad \Bigg( n^{\,\,\displaystyle{\prod_{j=1}^N \big(p_j^{k_j}-p_j^{k_j-1}\big)}} -1\Bigg) .$$

This completes the proof of Euler's Theorem.
\end{proof}

\section{Extra information}\label{sec:extra}

In this section we gather some  results we obtained so far (and come up with another one) in order to prove two  facts about the linear Alexander quandle we have been dealing with. We need the following Definitions.

\begin{definition}
The list of lengths of the disjoint cycles that make up a permutation is called the {\it pattern} of the permutation. The list of patterns of the permutations of a quandle is called the  {\it profile} of the quandle. If the profile, as a set, is a singleton, we identify the only pattern with the profile.
\end{definition}

\begin{definition}
A quandle $(X, \ast)$ is {\it connected} if for any $i, j\in X$ there are finitely many $k_1, \dots , k_l\in X$ such that $$j = ( \cdots ((i\ast k_1)\ast k_2 )\ast \cdots ) \ast k_l .$$
\end{definition}

\begin{theorem}\label{thm:linearAlexqdle}
Let $p$ be a prime. Let $n, k$ be positive integers with $n$ coprime with $p$.

Consider the linear Alexander quandle of order $n^{p^k}-1$ and parameter $n$. Then,
\begin{enumerate}[(i)]
\item the profile of this quandle is: $$\{ \underbrace{1, 1,  \dots , 1}_{\displaystyle{n-1}}, \underbrace{p, p,  \dots , p}_{\displaystyle{\frac{n^p-n}{p}}}, \dots , \underbrace{p^i, p^i,  \dots , p^i}_{\displaystyle{\frac{n^{p^i}-n^{p^{i-1}}}{p^i}}}, \dots , \underbrace{p^k, p^k,  \dots , p^k}_{\displaystyle{\frac{n^{p^k}-n^{p^{k-1}}}{p^k}}} \}$$

\item this quandle is not connected.
\end{enumerate}
\end{theorem}
\begin{proof}
\begin{enumerate}[(i)]
\item In the course of the proof of Theorem \ref{thm:flt+} (between Lemmas \ref{lem:iterperm} and \ref{lem:propersols}) we establish that the possible lengths of any disjoint cycle of any permutation of the quandle at issue are $$1, p, p^2, \dots , p^i, \dots, p^k .$$ In Lemma \ref{lem:propersols} we calculate the number of elements of the quandle that constitute all cycles for each of these possibilities: cycles of length $1$: $n-1$ elements; cycles of length $p$: $n^p-n$ elements; \dots    cycles of length $p^i$: $n^{p^i}-n^{p^{i-1}}$ elements; \dots   cycles of length $p^k$: $n^{p^k}-n^{p^{k-1}}$ elements. So there are $n-1$ cycles of length $1$; $\frac{n^p-n}{p}$ cycles of length $p$; \dots  $\frac{n^{p^i}-n^{p^{i-1}}}{p^i}$ cycles of length $p^i$; \dots $\frac{n^{p^k}-n^{p^{k-1}}}{p^k}$ cycles of length $p^k$.

\item We will show that successive quandle multiplication of $0$ by other elements will only give rise to multiples of $n-1$.
\begin{lemma}\label{lem:extra}
Let $(b_i)$ be a sequence of integers. Then, modulo $n^{p^k}-1$, we have $$(\cdots ((0\ast b_1)\ast b_2)\ast \cdots )\ast b_l = -(n-1)(n^{l-1}b_1 + n^{l-2}b_2 + \dots + nb_{l-1}+b_l) .$$
\end{lemma}
\begin{proof}
This is a straightforward exercise in induction that we leave to the reader.
\end{proof}
According to Lemma \ref{lem:extra}, there are no $b_1, \dots , b_l \in \mathbf{Z}$ such that $1= (\cdots ((0\ast b_1)\ast b_2 \dots )\ast b_l$ (mod $n^{p^k}-1$). This completes the proof of Theorem \ref{thm:linearAlexqdle}.
\end{enumerate}
\end{proof}

\section{Final remarks}\label{sec:final}

We would be remiss if we did not elaborate further about the totient function. As is known, the value of the totient function at $\prod_{j=1}^N  p_j^{k_j}$ is  $\prod_{j=1}^N \big(p_j^{k_j}-p_j^{k_j-1}\big)$, its meaning being the number of coprimes with  $\prod_{j=1}^N  p_j^{k_j}$, in the range between $1$ and $\prod_{j=1}^N  p_j^{k_j} -1$. But in the course of our proofs, we do not need the interpretation of the totient function. This is why the totient function has not been much mentioned this far.

\bigbreak

To the best of our knowledge, the algebraic proofs of Euler's Theorem found in the literature seem to be the following two (we let $\varphi$ stand for the totient function).

\begin{enumerate}[1.]

\item (\cite {Ireland_Rosen}, \cite{Jacobson}) Given a positive integer $r$, $\mathbf{Z}/r\mathbf{Z}$ is understood as a ring whose group of units has $\varphi(r)$ elements. Then using a corollary of Lagrange's theorem, any of these units (i.e., $[a]\in \mathbf{Z}/r\mathbf{Z}$ such that $a$ is coprime with $r$) satisfies $[a]^{\varphi(r)} = [1]$, modulo $r$. Thus, $r \mid (a^{\varphi(r)}-1)$.

\item (\cite{HardyWright}, \cite{Landau}) Given a positive integer $r$, $\mathbf{Z}/r\mathbf{Z}$ is understood as a ring whose group of units has $\varphi(r)$ elements. These $\varphi(r)$ elements have representatives in reduced form, as follows: $$1 \leq r_1 < r_2 < \dots < r_{\varphi(r)} < r .$$

    Upon multiplication by any
     unit $[a]\in \mathbf{Z}/r\mathbf{Z}$ (i.e., such that $a$ is coprime with $r$), we obtain: $$[a r_1] , [a r_2] , \dots , [a r_{\varphi(r)}]  .$$ Now, $$[a r_i] = [a r_j] \Longleftrightarrow [a] [r_i] = [a] [r_j] \Longleftrightarrow  [r_i] =  [r_j] \Longleftrightarrow  r_i = r_j .$$ Thus, $$\{[a r_1] , [a r_2] , \dots , [a r_{\varphi(r)}]\} = \{[r_1] , [r_2] , \dots , [r_{\varphi(r)}]\}  ,$$ so that $$\prod_{j=1}^{\varphi(r)} [r_j] = \prod_{j=1}^{\varphi(r)} [a r_j] = [a]^{\varphi(r)}\prod_{j=1}^{\varphi(r)} [r_j] \Longleftrightarrow [1] = [a]^{\varphi(r)} ,$$ i.e., $$r \mid (a^{\varphi(r)}-1) .$$

\end{enumerate}

Both of these arguments rely on the existence of a function that counts the units in $\mathbf{Z}/r\mathbf{Z}$ -- the totient function. Also, they rely on the structure of the group of units of the ring $\mathbf{Z}/r\mathbf{Z}$ and the first one also relies on Lagrange's Theorem.

On the other hand, our considerations set up the function $$\prod_{j=1}^N  p_j^{k_j} \mapsto \prod_{j=1}^N \big(p_j^{k_j}-p_j^{k_j-1}\big) ,$$ as the argument is developed. We do not  associate any ulterior meaning with this function. Furthermore, our proof has a simple interpretation related to counting the elements in the cycles with the relevant length (for the prime power case, Theorem \ref{thm:flt+}).

\bigbreak

\noindent
{\bf Acknowledgements}

Lages's work is supported by FCT (Funda\c c\~ao para a Ci\^encia e a Tecnologia), Portugal, through LisMath fellowship PD/BD/150348/2019.

Lopes acknowledges support by FCT (Funda\c c\~ao para a Ci\^encia e a Tecnologia), Portugal, through project FCT PTDC/MAT-PUR/31089/2017, ``Higher Structures and Applications''

Both authors acknowledge support  by FCT/Portugal through CAMGSD, IST-ID,
projects UIDB/04459/2020 and UIDP/04459/2020.

\end{document}